\documentclass[10pt,a4paper]{amsart}

\usepackage{amssymb}
\usepackage{hyperref}
\usepackage{enumitem}
\setlistdepth{5}

\newlist{myEnumerate}{enumerate}{5}
\setlist[myEnumerate,1]{label=(\arabic*)}
\setlist[myEnumerate,2]{label=(\alph*)}
\setlist[myEnumerate,3]{label=(\roman*)}
\setlist[myEnumerate,4]{label=(\Alph*)}
\setlist[myEnumerate,5]{label=(\Roman*)}

\newtheorem{thm}{Theorem}[section]

\newtheorem{lem}[thm]{Lemma}

\theoremstyle{definition}

\theoremstyle{remark}

\numberwithin{equation}{section}
\numberwithin{table}{section}

\newcommand{\Q}{\mathbb{Q}}  % The rational numbers.
\newcommand{\C}{\mathbb{C}}  % The complex numbers.
\newcommand{\h}{\mathbb{H}}  % The upper half plane numbers.
\newcommand{\Z}{\mathbb{Z}}  % The integers
\newcommand{\ord}{\mathcal{O}} % Order
\newcommand{\rcf}{\mathrm{RiCF}} %Ring class field
\newcommand{\gal}{\mathrm{Gal}} %Galois group

\begin{document}
	
	\title{Triples of singular moduli with rational product}
	\author{Guy Fowler}
	\address{Mathematical Institute, University of Oxford, 
		Oxford, OX2 6GG, United Kingdom.}
	\email{\href{mailto:guy.fowler@maths.ox.ac.uk}{guy.fowler@maths.ox.ac.uk}}
	\urladdr{\url{https://www.maths.ox.ac.uk/people/guy.fowler}}
	\date{\today}
	\thanks{\textit{Acknowledgements:} I would like to thank Jonathan Pila and Yuri Bilu for helpful comments and advice. This work was supported by an EPSRC doctoral scholarship.}
	%\subjclass[2010]{}
	
	\begin{abstract}
	We show that all triples $(x_1,x_2,x_3)$ of singular moduli satisfying $x_1 x_2 x_3 \in \Q^{\times}$ are ``trivial''. That is, either $x_1, x_2, x_3 \in \Q$; some $x_i \in \Q$ and the remaining $x_j, x_k$ are distinct, of degree $2$, and conjugate over $\Q$; or $x_1, x_2, x_3$ are pairwise distinct, of degree $3$, and conjugate over $\Q$. This theorem is best possible and is the natural three dimensional analogue of a result of Bilu, Luca, and Pizarro-Madariaga in two dimensions. It establishes an explicit version of the Andr\'e--Oort conjecture for the family of subvarieties $V_{\alpha} \subset \C^3$ defined by an equation $x_1 x_2 x_3 = \alpha \in \Q$. 
	\end{abstract}
	
	\maketitle
	
\section{Introduction}\label{sec:intro}

A singular modulus is the $j$-invariant of an elliptic curve over $\C$ with complex multiplication. Singular moduli arise precisely as those numbers of the form $x = j(\tau)$, where $\tau \in \h$ is such that $[\Q(\tau): \Q]=2$. Here $\h$ denotes the complex upper half plane and $j \colon \h \to \C$ is the modular $j$-function. 

Bilu, Luca, and Pizarro-Madariaga \cite{BiluLucaMadariaga16} proved the following result on non-zero rational products of singular moduli. (Note that since $0 = j(e^{ \pi i /3})$ is a singular modulus, we must exclude the case of product $0$ in order to obtain any kind of finiteness result.)

\begin{thm}[\cite{BiluLucaMadariaga16}]\label{thm:2d}
	Suppose $x_1, x_2$ are singular moduli such that $x_1 x_2 \in \Q^{\times}$. Then either $x_1, x_2 \in \Q^{\times}$ or $x_1, x_2$ are distinct, of degree $2$, and conjugate over $\Q$.
\end{thm}

Andr\'e \cite{Andre98} proved that an irreducible algebraic curve $V \subset \C^2$ contains only finitely many points $(x_1, x_2)$ with $x_1, x_2$ both singular moduli, unless $V$ is either a straight line or some modular curve $Y_0(N)$. Andr\'e's proof is ineffective, but effective versions of this theorem were subsequently proved by K\"uhne \cite{Kuhne12} and Bilu, Masser and Zannier \cite{BiluMasserZannier13}. Theorem~\ref{thm:2d} establishes an explicit version of Andr\'e's theorem for the class of hyperbolas given by equations $x_1 x_2 = \alpha$, where $\alpha \in \Q$.

 In this note, we prove the corresponding result on triples of singular moduli with non-zero rational product. 

\begin{thm}\label{thm:main}
	If $x_1, x_2, x_3$ are singular moduli such that $x_1 x_2 x_3 \in \Q^{\times}$, then one of the following holds:
	\begin{enumerate}
		\item $x_1, x_2, x_3 \in \Q^{\times}$;
		\item some $x_i \in \Q^{\times}$ and the remaining $x_j, x_k$ are distinct, of degree $2$, and conjugate over $\Q$;
		\item $x_1, x_2, x_3$ are pairwise distinct, of degree three, and conjugate over $\Q$.
	\end{enumerate}
\end{thm}

Conversely, if one of (1)--(3) holds, then it is clear that $x_1 x_2 x_3 \in \Q^{\times}$. Further, each of these cases is achieved. Theorem~\ref{thm:main} is therefore best possible. 

The aforementioned theorem of Andr\'e is an instance of the much more general Andr\'e--Oort conjecture on the special subvarieties of Shimura varieties. For subvarieties of $\C^n$, this conjecture was proved by Pila \cite{Pila11} (see therein for background on the conjecture also). Pila's proof, which uses o-minimality, is ineffective. Indeed, when $n > 2$ the conjecture for $\C^n$ is known effectively only for some very restricted classes of subvarieties, see \cite{BiluKuhne18} and \cite{Binyamini19}.

Theorem~\ref{thm:main} allows us to establish a completely explicit version of the Andr\'e--Oort conjecture for the following family of subvarieties of $\C^3$. Let $V_{\alpha}$ be the subvariety of $\C^3$ defined by the condition $x_1 x_2 x_3 = \alpha \in \Q$. The multiplicative independence modulo constants of pairwise distinct $\mathrm{GL}_2^+(\Q)$-translates of the $j$-function, proved by Pila and Tsimerman \cite[Theorem~1.3]{PilaTsimerman17}, shows that the only special subvarieties of $V_{\alpha}$ are given by imposing conditions of the form either $x_i = x_j$, or $x_k = \sigma$ for $\sigma$ a singular modulus. Our Theorem~\ref{thm:main} explicitly determines for each possible $\alpha$ which $\sigma$ occur in the definitions of the special subvarieties of $V_{\alpha}$. One thus obtains a fully explicit Andr\'e--Oort statement in this setting.

There are 13 rational singular moduli, one of which is $0$; 29 pairs of conjugate singular moduli of degree 2; and 25 triples of conjugate singular moduli of degree 3. The list of these may be computed in PARI. There are thus $364$ unordered triples $(x_1, x_2, x_3)$ as in (1); $348$ unordered triples $(x_1, x_2, x_3)$ as in (2); and $25$ unordered triples $(x_1, x_2, x_3)$ as in (3). One may straightforwardly compute the corresponding products $x_1 x_2 x_3$. There are 13 rational numbers which are the product of two distinct triples in (1) and 16 rational numbers which are the products of triples in (1) and (2). No other rational number is the product of more than one such triple. There are thus $708$ distinct non-zero rational numbers which arise as the product of three singular moduli, and the list of both these rational numbers and the corresponding triples of singular moduli is known. Since singular moduli are algebraic integers, we note that if $x_1 x_2 x_3 \in \Q$, then in fact $x_1 x_2 x_3 \in \Z$ and so these $708$ distinct rational numbers are all rational integers.

The plan of this note is as follows. Section~\ref{sec:background} contains the facts about singular moduli that we need for the proof of Theorem~\ref{thm:main}. The proof of Theorem~\ref{thm:main} is split over Sections~\ref{sec:pf} and \ref{sec:elim}. In Section~\ref{sec:pf}, we reduce to an effective finite list the possible triples $(x_1,x_2,x_3)$ of singular moduli with non-zero rational product which do not belong to one of the trivial cases (1)--(3) of Theorem~\ref{thm:main}. Then in Section~\ref{sec:elim} we explain how to use a PARI script \cite{PARI2} to eliminate all the triples on this list. The PARI scripts used in this article are available from \url{https://github.com/guyfowler/rationaltriples}.

\section{Background on singular moduli}\label{sec:background}

\subsection{Singular moduli and complex multiplication}\label{subsec:singmod}
We collect here those results about singular moduli which we will use in the sequel. For background on singular moduli and the theory of complex multiplication, see for example \cite{Cox89}.

Let $x$ be a singular modulus, so that $x = j(\tau)$ where $\tau \in \h$ is quadratic. Then $K = \Q(\tau)$ is an imaginary quadratic field, and one may write $K = \Q(\sqrt{d})$ for some square-free integer $d<0$. The singular modulus $x$ is the $j$-invariant of the CM elliptic curve $E_\tau = \C / \langle 1, \tau \rangle$, which has endomorphism ring $\ord = \mathrm{End}(E_\tau) \supsetneq \Z$. Here $\ord$ is an order in the imaginary quadratic field $K$.

One associates to $x$ its discriminant $\Delta$, which is the discriminant of the order $\ord = \mathrm{End}(E_\tau)$. One has that $\Delta = f^2 D$, where $D$ is the discriminant of the number field $K$ (the fundamental discriminant) and $f = [\ord_{K} : \ord]$ is the conductor of the order $\ord$ (here $\ord_{K}$ is the ring of integers of $K$). One has that $K = \Q(\sqrt{\Delta})$. Further, $\Delta = b^2-4ac$, where $a, b, c \in \Z$ are such that $a\tau^2 + b \tau + c =0$ and $\gcd(a,b,c)=1$.

The singular moduli of a given discriminant $\Delta$ form a full Galois orbit over $\Q$, and one has that $[\Q(x) : \Q] = h(\Delta)$, where $h(\Delta)$ is the class number of the (unique) imaginary quadratic order of discriminant $\Delta$. The Galois group of $\Q(x)$ acts sharply transitively on the singular moduli of discriminant $\Delta$.

For a discriminant $\Delta$, we define $H_{\Delta}$, the Hilbert class polynomial of discriminant $\Delta$, by
\[H_{\Delta}(x) = \prod_{i=1}^n (x - x_i),\]
where $x_1, \ldots, x_n$ are the singular moduli of discriminant $\Delta$. The polynomial $H_{\Delta}$ has integer coefficients and is irreducible over the field $K = \Q(\sqrt{\Delta})$. The splitting field of $H_{\Delta}$ over $K$ is equal to $K(x_i)$ for $i=1,\ldots,n$. The field $K(x_i)$ is the ring class field of the imaginary quadratic order $\mathcal{O}$ of discriminant $\Delta$. Thus $K(x_i) / K$ is an abelian extension with $\gal(K(x_i)/K) \cong \mathrm{cl}(\mathcal{O})$, the ideal class group of the order $\mathcal{O}$. Hence, $[K(x_i) : K] = h(\Delta)$ and the singular moduli of discriminant $\Delta$ form a full Galois orbit over $K$ also.

The singular moduli of a given discriminant $\Delta$ may be explicitly described in the following way \cite[Proposition~2.5]{BiluLucaMadariaga16}. Write $T_\Delta$ for the set of triples $(a,b,c) \in \Z^3$ such that: $\gcd(a,b,c)=1$, $\Delta = b^2-4ac$, and either $-a < b \leq a < c$ or $0 \leq b \leq a = c$. Then there is a bijection between $T_\Delta$ and the singular moduli of discriminant $\Delta$, given by $(a,b,c) \mapsto j((b + \sqrt{\Delta})/2a)$.  For a singular modulus $x$ of discriminant $\Delta$, one thus has that $\lvert T_{\Delta} \rvert = [\Q(x) : \Q] = [K(x) : K]= h(\Delta)$, where $K= \Q(\sqrt{\Delta})$.

\subsection{Bounds on singular moduli}\label{subsec:bounds}

We now state some upper and lower bounds for singular moduli. These bounds will be used without special reference in Section~\ref{sec:pf}. For every non-zero singular modulus $x$ of discriminant $\Delta$, we have (\cite[(12)]{BiluLucaMadariaga16}) the lower bound 
\[ \lvert x \rvert \geq \min\{4.4 \times 10^{-5}, 3500 \lvert \Delta \rvert^{-3}\}.\]

The $j$-function has a Fourier expansion $j(z) = \sum_{n=-1}^{\infty} c_n q^n$ in terms of the nome $q=e^{2 \pi i z}$ with $c_n \in \Z_{>0}$ for all $n$. Consequently, for a singular modulus $x$ corresponding to a triple $(a, b, c) \in T_{\Delta}$ we have (\cite[\S2]{FayeRiffaut18}):
\[e^{\pi \lvert \Delta \rvert^{1/2}/a}-2079 \leq \lvert x \rvert \leq e^{\pi \lvert \Delta \rvert^{1/2}/a}+2079.\]
We will apply this bound variously with $a=2,3,4,5$, making use of Lemma~\ref{lem:dom} below. For a singular modulus $x$ corresponding to a triple $(a,b,c) \in T_{\Delta}$ with $a=1$ and $\lvert \Delta \rvert \geq 23$, one obtains also that $ \lvert x \rvert \geq 0.9994 e^{\pi \lvert \Delta \rvert^{1/2}}$, see \cite[(11)]{BiluLucaMadariaga16}. 

\begin{lem}\label{lem:dom}
	For a given discriminant $\Delta$, there exists:
	\begin{enumerate}
		\item a unique singular modulus, corresponding to a triple with $a=1$;
		\item at most two singular moduli, corresponding to triples with $a=2$, and if $\Delta \equiv 4 \bmod 16$, then there are no such singular moduli;
		\item at most two singular moduli corresponding to triples with $a=3$;
		\item at most two singular moduli corresponding to triples with $a=4$;
		\item at most two singular moduli corresponding to triples with $a=5$.
\end{enumerate}
\end{lem}

\begin{proof}
	The first two claims are Proposition~2.6 of \cite{BiluLucaMadariaga16}. We show the remaining claims.
	
	Suppose $a=3$. Let $(3, b_1, c_1), (3, b_2, c_2)$ be two such tuples. Then $b_1, b_2 \in \{ -2, -1, 0, 1, 2, 3\}$. Since $\Delta = b^2 - 4ac$ for all $(a,b,c) \in T_{\Delta}$, one has that $b_1^2 -12c_1 = b_2^2 -12c_2$. Thus $b_1^2 - b_2^2 \equiv 0 \bmod 12$. Therefore, it must be that $b_1 = \pm b_2$. Since $a_i, b_i$ together uniquely determine $c_i$, there are at most two tuples in $T_{\Delta}$ with $a=3$.
	
	Now let $a=4$. Suppose $(4, b_1, c_1), (4, b_2, c_2)$ are two such tuples. Then $b_1, b_2 \in \{ -3, -2, -1, 0, 1, 2, 3, 4\}$. Since $\Delta = b^2 - 4ac$ for all $(a,b,c) \in T_{\Delta}$, one has that $b_1^2 -16c_1 = b_2^2 -16c_2$. Thus $b_1^2 - b_2^2 \equiv 0 \bmod 16$. Therefore, it must be that either $b_1 = \pm b_2$ or $\{b_1, b_2\}=\{0,4\}$. Since $a_i, b_i$ together uniquely determine $c_i$, there are at most two tuples in $T_{\Delta}$ with $a=4$.
	
	Let $a=5$. Suppose $(5, b_1, c_1), (5, b_2, c_2)$ are two such tuples. Then $b_1, b_2 \in \{ -4, -3, -2, -1, 0, 1, 2, 3, 4, 5\}$. Since $\Delta = b^2 - 4ac$ for all $(a,b,c) \in T_{\Delta}$, one has that $b_1^2 -20c_1 = b_2^2 -20c_2$. Thus $b_1^2 - b_2^2 \equiv 0 \bmod 20$. Therefore, it must be that $b_1 = \pm b_2$. Since $a_i, b_i$ together uniquely determine $c_i$, there are at most two tuples in $T_{\Delta}$ with $a=5$.
\end{proof}

\subsection{Fields generated by singular moduli}\label{subsec:fields}

Our proof of Theorem~\ref{thm:main} will also rely on some results about the fields generated by singular moduli. The first of these is a result on when two singular moduli generate the same field. It was proved mostly in \cite{AllombertBiluMadariaga15}, as Corollary~4.2 and Proposition~4.3. For the ``further'' claim in (2), see \cite[\S3.2.2]{BiluLucaMadariaga16}.

\begin{lem}\label{lem:samefield}
	Let $x_1, x_2$ be singular moduli with discriminants $\Delta_1, \Delta_2$ respectively. Suppose that $\Q(x_1) = \Q(x_2)$, and denote this field $L$. Then $h(\Delta_1) = h(\Delta_2)$, and we have that:
	\begin{enumerate}
		\item If $\Q(\sqrt{\Delta_1}) \neq \Q(\sqrt{\Delta_2})$, then the possible fields $L$ are listed in \cite[Table~4.1]{AllombertBiluMadariaga15}. Further, the field $L$ is Galois and the discriminant of any singular modulus $x$ with $\Q(x) = L$ is also listed in this table.
		\item If $\Q(\sqrt{\Delta_1}) = \Q(\sqrt{\Delta_2})$, then either: $L = \Q$ and $\Delta_1, \Delta_2 \in \{-3, -12, -27\}$; or: $\Delta_1 / \Delta_2 \in \{1, 4, 1/4\}$. Further, if $\Delta_1 = 4 \Delta_2$, then $\Delta_2 \equiv 1 \bmod 8$.
	\end{enumerate}
\end{lem}

We now establish a similar result on when one singular modulus generates a degree $2$ subfield of the field generated by another singular modulus. We split the proof into the next two lemmas.

\begin{lem}\label{lem:subfieldsamefund}
	Let $x_1, x_2$ be singular moduli such that $[\Q(x_1) : \Q(x_2)]=2$. Suppose that $\Q(\sqrt{\Delta_1}) = \Q(\sqrt{\Delta_2})$. Then either $\Delta_1 \in \{9 \Delta_2 / 4, 4\Delta_2, 9 \Delta_2, 16\Delta_2\}$, or $x_2 \in \Q$.	
\end{lem}

\begin{proof}
	The argument is a modification of the proof in \cite{AllombertBiluMadariaga15} of Lemma~\ref{lem:field}. Given an imaginary quadratic field $K$ of (fundamental) discriminant $D$ and an integer $f \geq 1$, we write $\rcf(K,f)$ for the ring class field of the imaginary quadratic order of discriminant $\Delta = f^2 D$.
	
	Let $x_1, x_2$ be singular moduli satisfying the hypotheses of the lemma. Denote $K$ the field $\Q(\sqrt{\Delta_1}) = \Q(\sqrt{\Delta_2})$. The singular moduli $x_1, x_2$ have the same fundamental discriminant (the discriminant of the field $K$), which we denote $D$. We may then write $\Delta_1 = f_1^2 D$, $\Delta_2 = f_2^2 D$. Note that $\Delta_1 \neq \Delta_2$. Let $f=\mathrm{lcm}(f_1,f_2)$.
	
	Suppose in addition that $D \neq -3,-4$. Then \cite[Proposition~3.1]{AllombertBiluMadariaga15}
	\[\rcf(K,f_1)\rcf(K,f_2)=\rcf(K,f),\]
	where the left hand side denotes the compositum of $\rcf(K,f_1)$ and $\rcf(K,f_2)$. By the theory of complex multiplication $\rcf(K,f_i)=K(x_i)$ for $i=1,2$. Thus $\rcf(K,f)=K(x_1)$ since $x_2 \in \Q(x_1)$. Therefore
	\[h(f^2 D)=[\rcf(K,f) : K]=[\rcf(K,f_1) : K]= h(f_1^2 D).\]
	Also, 
	\[  [\Q(x_1) : \Q] = 2 [\Q(x_2) : \Q],\] 
	and thus $h(f^2 D)=h(f_1^2 D) = 2h(f_2^2 D)$. As in the proof of \cite[Proposition~4.3]{AllombertBiluMadariaga15}, one may then use the class number formula \cite[(6)]{AllombertBiluMadariaga15} to obtain that 
	\[\frac{f}{f_1} \prod_{\substack{p \mid f\\ p \nmid f_1}} (1-\Bigl(\frac{D}{p}\Bigr) \frac{1}{p}) = 1\]
	and 
	\[\frac{f}{f_2} \prod_{\substack{p \mid f\\ p \nmid f_2}} (1-\Bigl(\frac{D}{p}\Bigr) \frac{1}{p}) = 2,\]
	where $\bigl(\frac{D}{p}\bigr)$ denotes the Kronecker symbol.
	This implies that $f/f_1 \in \{1,2\}$ and $f/f_2 \in \{2,3,4\}$. One thus has that $f_1/f_2 \in \{3/2,2,3,4\}$ since $f_1 \neq f_2$. Hence $\Delta_1 \in \{9 \Delta_2/4, 4 \Delta_2, 9 \Delta_2, 16 \Delta_2\}$.
	
	Now let $D \in \{-3,-4\}$. If $\gcd(f_1,f_2)>1$, then, by \cite[Proposition~3.1]{AllombertBiluMadariaga15} again,
		\[\rcf(K,f_1)\rcf(K,f_2)=\rcf(K,f),\]
		and the above proof works. If $f_1 =1$, then $\Q(x_1)=\Q$, a contradiction. If $f_2 =1$, then $x_2 \in \Q$ and we are done. 
				
		So we may now assume that $f_1, f_2 >1$ and $\gcd(f_1,f_2)=1$. So $f=f_1 f_2$. In this case, by \cite[Proposition~3.1]{AllombertBiluMadariaga15},
		\[2h(f_2^2 D) = h(f_1^2 D) = l^{-1} h(f^2 D),\]
		where $l=2$ for $D=-4$ and $l=3$ for $D=-3$. We now apply again the class number formula to obtain that
		\[f_1 \prod_{p \mid f_1} (1-\Bigl(\frac{D}{p}\Bigr) \frac{1}{p}) = 2l\]
		and 
		\[f_2 \prod_{p \mid f_2} (1-\Bigl(\frac{D}{p}\Bigr) \frac{1}{p}) = l.\]
	 Inspecting the possibilities for $f_2$, we see that 
		\[ \Delta_2 \in \{-12,-16,-27\}.\]
		 But then $x_2 \in \Q$.
		\end{proof}
	
	\begin{lem}\label{lem:subfielddiffund}
		Let $x_1, x_2$ be singular moduli such that $[\Q(x_1) : \Q(x_2)]=2$. Suppose that $\Q(\sqrt{\Delta_1}) \neq \Q(\sqrt{\Delta_2})$. Then one of the following holds:
		\begin{enumerate}
			\item at least one of $\Delta_1$ or $\Delta_2$ is listed in \cite[Table~2.1]{AllombertBiluMadariaga15} and the corresponding field $\Q(x_i)$ is Galois;
			\item  $h(\Delta_1) \geq 128$.
		\end{enumerate}
		\end{lem}
		
		\begin{proof}
		This proof is a modified version of \cite[Theorem~4.1]{AllombertBiluMadariaga15}. If $\Q(x_1)$ is Galois over $\Q$, then by Corollaries~3.3 and 2.2 and Remark~2.3 of \cite{AllombertBiluMadariaga15}, either $\Delta_1$ is listed in \cite[Table~2.1]{AllombertBiluMadariaga15} or $h(\Delta_1) \geq 128$. If $\Q(x_2)$ is Galois over $\Q$, then similarly either $\Delta_2$ is listed in \cite[Table~2.1]{AllombertBiluMadariaga15} or $h(\Delta_2) \geq 128$ (and so certainly $h(\Delta_1) \geq 128$). 
		
		So we may now suppose that neither $\Q(x_1)$ nor $\Q(x_2)$ is Galois over $\Q$. We will show this leads to a contradiction. Let $M_1$ be the Galois closure of $\Q(x_1)$ over $\Q$. Then $M_1= \Q(\sqrt{\Delta_1},x_1) \supset \Q(x_2)$. Let $M_2$ be the Galois closure of $\Q(x_2)$ over $\Q$. Then $M_2 = \Q(\sqrt{\Delta_2}, x_2)$. Also $M_2 \subset M_1$ since $M_1$ is Galois and contains $\Q(x_2)$. Since $\Q(x_1), \Q(x_2)$ are not Galois, one has that $\sqrt{\Delta_i} \notin \Q(x_i)$. Hence $[M_1 : \Q]= 2 h(\Delta_1)$ and $[M_2 : \Q]= 2 h(\Delta_2)$. In particular, $[M_1 : M_2]=2$ since $h(\Delta_1) = 2 h(\Delta_2)$.
		
		Let $G = \gal(M_1/\Q)$, $H=\gal(M_1/\Q(\sqrt{\Delta_1},\sqrt{\Delta_2}))$, and $H_i =  \gal(M_1/\Q(\sqrt{\Delta_i}))$ for $i=1,2$. So $H = H_1 \cap H_2$, $[H_1 : H]=2$, and $[H_2 : H]=2$. As in the proof of \cite[Theorem~4.1]{AllombertBiluMadariaga15}, one has that $H$ is isomorphic to $(\Z / 2\Z)^n$ for some $n$. Each of $H_1, H_2$ contains $H$ as an index $2$ subgroup. So $H_1, H_2$ must each be isomorphic to either $(\Z / 2\Z)^{n+1}$ or $(\Z / 4 \Z) \times (\Z / 2\Z)^{n-1}$. If $H_1 \cong (\Z / 2\Z)^{n+1}$, then by \cite[Corollary~3.3]{AllombertBiluMadariaga15}, the field $\Q(x_1)$ is Galois, a contradiction. So we may assume that $H_1 \cong (\Z / 4 \Z) \times (\Z / 2\Z)^{n-1}$.
		
		Suppose next that $H_2 \cong (\Z / 2\Z)^{n+1}$. Observe that $H_2$ is abelian, and hence all its subgroups are normal. We have the extension of fields $M_1 / \Q(\sqrt{\Delta_2}) / \Q$, where $M_1 / \Q$ is Galois, and so the extension $M_1 / \Q(\sqrt{\Delta_2})$ is also Galois and its Galois group is $H_2$. Therefore, considering the extension $M_1 / M_2 / \Q(\sqrt{\Delta_2})$, we obtain that $\gal(M_2 / \Q(\sqrt{\Delta_2}))$ is isomorphic to the quotient $H_2 / \gal(M_1 / M_2)$. Note that $\gal(M_1 / M_2) \leq H_2$ and $\lvert \gal(M_1 / M_2) \rvert = 2$. Thus, since $H_2 \cong (\Z / 2\Z)^{n+1}$, we must have that $\gal(M_1 /M_2) \cong (\Z / 2\Z)$ and $\gal(M_2 / \Q(\sqrt{\Delta_2})) \cong (\Z / 2 \Z)^n$. This implies, by \cite[Corollary~3.3]{AllombertBiluMadariaga15} again, that $\Q(x_2)$ is Galois, a contradiction.
		
		Hence we must have that $H_1 \cong H_2 \cong (\Z / 4 \Z) \times (\Z / 2\Z)^{n-1}$. Exactly as in \cite[Theorem~4.1]{AllombertBiluMadariaga15}, this implies that $G \cong D_8 \times (\Z / 2\Z)^{n-1}$, where $D_8$ denotes the dihedral group with $8$ elements. The group $D_8 \times (\Z / 2\Z)^{n-1}$ has only one subgroup isomorphic to $(\Z / 4 \Z) \times (\Z / 2\Z)^{n-1}$, and hence one must have that $H_1 = H_2$. This though implies that $\Q(\sqrt{\Delta_1})=\Q(\sqrt{\Delta_2})$, a contradiction. 
	\end{proof}

The other result we use is on the fields generated by products of pairs of non-zero singular moduli, and establishes that such a field is ``close to'' the field generated by the pair of singular moduli. This result is due to Faye and Riffaut \cite{FayeRiffaut18}. 

\begin{lem}[{\cite[Theorem~1.3]{FayeRiffaut18}}]\label{lem:field}
	Let $x_1, x_2$ be distinct non-zero singular moduli of discriminants $\Delta_1, \Delta_2$ respectively. Then $\Q(x_1x_2) = \Q(x_1, x_2)$ unless $\Delta_1 = \Delta_2$, in which case $[\Q(x_1, x_2) : \Q(x_1x_2)] \leq 2$.
\end{lem}

\section{An effective bound}\label{sec:pf}
We now turn to the proof of Theorem~\ref{thm:main} itself. Suppose $x_1, x_2, x_3$ are singular moduli such that $x_1 x_2 x_3 = \alpha \in \Q^{\times}$. Write $\Delta_i$ for their respective discriminants and $h_i$ for the corresponding class numbers $h(\Delta_i)$. In this section, we will reduce the possible $(\Delta_1, \Delta_2, \Delta_3)$ to an (effectively) finite list. Without loss of generality, assume that $h_1 \geq h_2 \geq h_3$.  

If the $x_i$ are not pairwise distinct, then (1) of Theorem~\ref{thm:main} must hold by a theorem of Riffaut \cite[Theorem~1.6]{Riffaut19}, which classifies all pairs $(x,y)$ of singular moduli satisfying $x^m y^n \in \Q^{\times}$ for some $m, n \in \Z \setminus \{0\}$. So we may and do assume that $x_1, x_2, x_3$ are pairwise distinct. If $h_3=1$, then $x_3 \in \Q^{\times}$ and hence $x_1 x_2 \in \Q^{\times}$. Thus by the two dimensional case Theorem~\ref{thm:2d}, proved in \cite{BiluLucaMadariaga16}, either $x_1, x_2 \in \Q$ or $x_1, x_2$ are of degree $2$ and conjugate over $\Q$. Thus either (1) or (2) in Theorem~\ref{thm:main} holds and we are done. We therefore assume subsequently that $h_1 \geq h_2 \geq h_3 \geq 2$.

Clearly we have that $\Q(x_1) = \Q(x_2 x_3)$. Thus,
by Lemma~\ref{lem:field}, we have that
\[[\Q(x_1) : \Q] = [\Q(x_2 x_3) : \Q]=  \begin{cases} \mbox{ either } [\Q(x_2, x_3) : \Q], \\
\mbox{ or } \frac{1}{2}[\Q(x_2, x_3) : \Q].
\end{cases}\]
Noting that $h_2 = [\Q(x_2) : \Q]$ and $h_3 = [\Q(x_3) : \Q]$ each divide $[\Q(x_2,x_3) : \Q]$, we must have that $h_2, h_3 \mid 2 [\Q(x_1) : \Q]$. Hence $h_2, h_3 \mid 2 h_1$. Symmetrically we have also that $h_1, h_2 \mid 2h_3$ and $h_1, h_3 \mid 2 h_2$.  Then, since $h_1 \geq h_2 \geq h_3$, one of the following must hold: either $h_1 = h_2 = h_3$, or $h_1 = h_2 = 2h_3$, or $h_1 = 2 h_2 = 2 h_3$. We consider each of these cases in turn.

We will write $(x_1', x_2', x_3')$ for a conjugate of $(x_1, x_2, x_3)$, where $x_i'$ is the conjugate of $x_i$ associated to an element $(a_i', b_i', c_i') \in T_{\Delta_i}$. Computations in this section were carried out in PARI \cite{PARI2}.

\subsection{The case $h_1 = h_2 = h_3$}\label{subsec:equal}
Write $h = h_1 = h_2 = h_3$. We split this situation into subcases, depending as to whether the $\Delta_i$ are equal.

\subsubsection{The subcase $\Delta_1 = \Delta_2 = \Delta_3$}\label{subsubsec:equaldisc}
Write $\Delta$ for this shared discriminant. The $x_i$ are thus all singular moduli of discriminant $\Delta$ and hence are all conjugate. Since the $x_i$ are pairwise distinct, they must then be of degree at least $3$, so $h \geq 3$. If $h=3$, then we are in case (3) of Theorem~\ref{thm:main}. So we may assume that $h \geq 4$. 

Taking conjugates as necessary, we may assume that $x_1$ is dominant. Since $h \geq 4$, one certainly has $\lvert \Delta \rvert \geq 23$. Thus, by the bounds in Subsection~\ref{subsec:bounds}, one has the lower bound for $\lvert \alpha \rvert$ given by
\[\lvert \alpha \rvert \geq (0.9994 e^{\pi \lvert \Delta \rvert^{1/2}}) (\min\{4.4 \times 10^{-5}, 3500 \lvert \Delta \rvert^{-3}\})^2.\]

We now establish upper bounds for $\lvert \alpha \rvert$, incompatible with this lower bound for suitably large $\lvert \Delta \rvert$. The larger the class number $h$, the better these bounds will be. Let $\sigma_1, \ldots, \sigma_h$ be the automorphisms of $\Q(x_1)$. Then $\sigma_i(x_1, x_2, x_3) = (x_1', x_2', x_3')$, where $x_1', x_2', x_3'$ are themselves singular moduli of discriminant $\Delta$, since these singular moduli form a complete Galois orbit over $\Q$. Further, if $\sigma_i(x_k) = \sigma_j(x_k)$, then $i=j$ since the action is sharply transitive. Thus each singular moduli of discriminant $\Delta$ occurs at most once among the $\sigma_i (x_k)$ for $i=1, \ldots, h$.

Let $k, m_1, m_2, m_3 \in \Z$ be as given in the following table. When $h \geq k$, we can by Lemma~\ref{lem:dom} find a conjugate $(x_1', x_2', x_3')$ of $(x_1, x_2, x_3)$, where each $x_i'$ is a singular modulus corresponding to a tuple $(a_i', b_i', c_i') \in T_{\Delta}$ with $a_i' \geq m_i$. 

	\vspace{0.6cm}

\begin{center}
	\begin{tabular}{  c | c | c | c }\label{tab:conj}
		
			$m_1$ & $m_2$ & $m_3$ & $k$ \\ \hline
			$3$ & $3$ & $4$ & $12$\\
			$3$ & $4$ & $4$ & $14$\\
			$4$ & $4$ & $4$ & $16$\\
			$4$ & $4$ & $5$ & $18$\\
			$4$ & $5$ & $5$ & $20$\\
	\end{tabular}
\end{center}

\vspace{0.6cm}

Since $x_1' x_2' x_3' = \alpha$, each such conjugate gives rise to an upper bound for $\lvert \alpha \rvert$ of the form
\[ \lvert \alpha \rvert \leq (e^{\pi \lvert \Delta \rvert^{1/2}/m_1}+2079)(e^{\pi \lvert \Delta \rvert^{1/2}/m_2}+2079)(e^{\pi \lvert \Delta \rvert^{1/2}/m_3}+2079).\]
For $(m_1, m_2, m_3)$ as in the above table, these bounds are incompatible with the earlier lower bound for $\lvert \alpha \rvert$ when $\lvert \Delta \rvert$ is suitably large. Explicitly, we obtain that one of the following holds:
\begin{enumerate}
	\item $4 \leq h \leq 11$;
	\item $12 \leq h \leq 13$ and $\lvert \Delta \rvert \leq 30339$;\footnote{We observe that the bound on $\lvert \Delta \rvert$ obtained here is superfluous, since in fact $h \leq 13$ already implies $\lvert \Delta \rvert \leq 20563$, as may be demonstrated in Sage \cite{sagemath} using the function cm\_orders().}
	\item $14 \leq h \leq 15$ and $\lvert \Delta \rvert \leq 4124$;
	\item $16 \leq h \leq 17$ and $\lvert \Delta \rvert \leq 1045$;
	\item $18 \leq h \leq 19$ and $\lvert \Delta \rvert \leq 488$;
	\item $20 \leq h$ and $\lvert \Delta \rvert \leq 334$.
\end{enumerate}

\subsubsection{The subcase where the $\Delta_i$ are not all equal}\label{subsubsec:unequaldisc}
Without loss of generality assume that $\lvert \Delta_1 \rvert > \lvert \Delta_2 \rvert$. Then $\Q(x_3) = \Q(x_1x_2) = \Q(x_1, x_2)$, where the last equality holds by Lemma~\ref{lem:field} since $\Delta_1 \neq \Delta_2$. Thus $\Q(x_1), \Q(x_2) \subset \Q(x_3)$. Since $h_1 = h_2 = h_3$, these inclusions are in fact equalities. Thus $\Q(x_1)= \Q(x_2) = \Q(x_3)$. Denote this field $L$.

Suppose first that $\Q(\sqrt{\Delta_i}) \neq \Q(\sqrt{\Delta_j})$ for some $i, j$. Then by (1) of Lemma~\ref{lem:samefield}, we have that the field $L$ is listed in \cite[Table~4.1]{AllombertBiluMadariaga15}, as are the possible discriminants $\Delta_1, \Delta_2, \Delta_3$.

So we reduce to the situation where $\Q(\sqrt{\Delta_1}) = \Q(\sqrt{\Delta_2}) = \Q(\sqrt{\Delta_3})$. Then by (2) of Lemma~\ref{lem:samefield} either $h=1$ or, for every $i, j$, we have that $\Delta_i / \Delta_j \in \{1, 4, 1/4\}$. Since $h \geq 2$, we must be in the second case. Write $\Delta = \Delta_2$. Then we have that $\Delta \equiv 1 \bmod 8$, $\Delta_1 = 4 \Delta_2 = 4 \Delta$, and either $\Delta_3 = \Delta$ or $\Delta_3 = 4 \Delta = \Delta_1$. Also, $\lvert \Delta_1 \rvert \geq 23$ since $\Delta_1 = 4 \Delta_2$ and $h \geq 2$ implies $\lvert \Delta_2 \rvert \geq 15$.

Suppose first that $\Delta_3 = \Delta$. Taking conjugates, assume that $x_1$ is dominant. Then by the bounds in Subsection~\ref{subsec:bounds} we have the lower bound
\begin{align*}
	\lvert \alpha \rvert \geq &(0.9994 e^{\pi \lvert \Delta_1 \rvert^{1/2}}) (\min\{4.4 \times 10^{-5}, 3500 \lvert \Delta_2 \rvert^{-3}\})\\
	&(\min\{4.4 \times 10^{-5}, 3500 \lvert \Delta_3 \rvert^{-3}\}),\\
	\geq &(0.9994 e^{2 \pi \lvert \Delta \rvert^{1/2}}) (\min\{4.4 \times 10^{-5}, 3500 \lvert \Delta \rvert^{-3}\})^2.
\end{align*}

Since $\Q(x_1) = \Q(x_2) = \Q(x_3)$, the Galois orbit of $(x_1, x_2, x_3)$ has exactly $h$ elements. Each conjugate of $x_i$ occurs exactly once as the $i$th coordinate of a conjugate $(x_1', x_2', x_3')$ of $(x_1, x_2, x_3)$. Since $\Delta \equiv 1 \bmod 8$ and $\Delta_1 = 4 \Delta$, there are no tuples $(a,b,c) \in T_{\Delta_1}$ with $a=2$ because $\Delta_1 \equiv 4 \bmod 16$. Let $k, m_1, m_2, m_3 \in \Z$ be as given in the following table. When $h \geq k$, we can by Lemma~\ref{lem:dom} find a conjugate $(x_1', x_2', x_3')$ of $(x_1, x_2, x_3)$, where each $x_i'$ is a singular modulus corresponding to a tuple $(a_i', b_i', c_i') \in T_{\Delta_i}$ with $a_i' \geq m_i$.

\vspace{0.6cm}

\begin{center}
	\begin{tabular}{  c | c | c | c }\label{tab:conj}
		
		$m_1$ & $m_2$ & $m_3$ & $k$ \\ \hline
		$3$ & $2$ & $2$ & $4$\\
		$3$ & $3$ & $2$ & $6$\\
		$3$ & $3$ & $3$ & $8$\\
	\end{tabular}
\end{center}

\vspace{0.6cm}

Since $x_1' x_2' x_3' = \alpha$, each such conjugate gives rise to an upper bound for $\lvert \alpha \rvert$ of the form
\begin{align*}
	 \lvert \alpha \rvert &\leq (e^{\pi \lvert \Delta_1 \rvert^{1/2}/m_1}+2079)(e^{\pi \lvert \Delta_2 \rvert^{1/2}/m_2}+2079)(e^{\pi \lvert \Delta_3 \rvert^{1/2}/m_3}+2079)\\
	 &= (e^{2\pi \lvert \Delta \rvert^{1/2}/m_1}+2079)(e^{\pi \lvert \Delta \rvert^{1/2}/m_2}+2079)(e^{\pi \lvert \Delta \rvert^{1/2}/m_3}+2079).
	 \end{align*}
For $(m_1, m_2, m_3)$ as in the above table, these bounds are incompatible with the earlier lower bound for $\lvert \alpha \rvert$ when $\lvert \Delta \rvert$ is suitably large. We obtain therefore that one of the following must hold:
\begin{enumerate}
	\item $2 \leq h \leq 3$;
	\item $4 \leq h \leq 5$ and $\lvert \Delta \rvert \leq 367$;
	\item $6 \leq h \leq 7$ and $\lvert \Delta \rvert \leq 163$;
	\item $8 \leq h$ and $\lvert \Delta \rvert \leq 93$.
\end{enumerate}

Now suppose that $\Delta_1=\Delta_3=4\Delta_2$, where $\Delta_2 = \Delta \equiv 1 \bmod 8$. Taking conjugates, assume that $x_1$ is dominant. Then by the bounds in Subsection~\ref{subsec:bounds} we have the lower bound
\begin{align*}
	\lvert \alpha \rvert \geq &(0.9994 e^{\pi \lvert \Delta_1 \rvert^{1/2}}) (\min\{4.4 \times 10^{-5}, 3500 \lvert \Delta_2 \rvert^{-3}\})\\
	&(\min\{4.4 \times 10^{-5}, 3500 \lvert \Delta_3 \rvert^{-3}\}),\\
\geq &(0.9994 e^{2 \pi \lvert \Delta \rvert^{1/2}}) (\min\{4.4 \times 10^{-5}, 3500 \lvert \Delta \rvert^{-3}\})\\
& (\min\{4.4 \times 10^{-5}, 3500 \times 4^{-3} \lvert \Delta \rvert^{-3}\}).
\end{align*}

Since $\Delta \equiv 1 \bmod 8$ and $\Delta_1 = \Delta_3= 4 \Delta$, as before there are no tuples $(a,b,c) \in T_{\Delta_1} = T_{\Delta_3}$ with $a=2$. Therefore, when $h \geq k$, we can by Lemma~\ref{lem:dom} find conjugates $(x_1', x_2', x_3')$ of $(x_1, x_2, x_3)$, where each $x_i'$ is a singular modulus corresponding to a tuple $(a_i', b_i', c_i') \in T_{\Delta_i}$ with $a_i' \geq m_i$, and $k, m_1, m_2, m_3 \in \Z$ given in the following table. 

\vspace{0.6cm}

\begin{center}
	\begin{tabular}{  c | c | c | c }\label{tab:conj}
		
		$m_1$ & $m_2$ & $m_3$ & $k$ \\ \hline
		$3$ & $2$ & $3$ & $4$\\
		$3$ & $3$ & $3$ & $6$\\
		$4$ & $3$ & $3$ & $8$\\
			$4$ & $3$ & $4$ & $10$
	\end{tabular}
\end{center}

\vspace{0.6cm}

Each such conjugate gives rise to an upper bound for $\lvert \alpha \rvert$ of the form
\begin{align*}
	\lvert \alpha \rvert &\leq (e^{\pi \lvert \Delta_1 \rvert^{1/2}/m_1}+2079)(e^{\pi \lvert \Delta_2 \rvert^{1/2}/m_2}+2079)(e^{\pi \lvert \Delta_3 \rvert^{1/2}/m_3}+2079)\\
	&= (e^{2\pi \lvert \Delta \rvert^{1/2}/m_1}+2079)(e^{\pi \lvert \Delta \rvert^{1/2}/m_2}+2079)(e^{2 \pi \lvert \Delta \rvert^{1/2}/m_3}+2079).
\end{align*}
We thus obtain that one of the following holds:
\begin{enumerate}
	\item $2 \leq h \leq 3$;
	\item $4 \leq h \leq 5$ and $\lvert \Delta \rvert \leq 5781$;
	\item $6 \leq h \leq 7$ and $\lvert \Delta \rvert \leq 650$;
	\item $8 \leq h \leq 9$ and $\lvert \Delta \rvert \leq 192$;
	\item $10 \leq h$ and $\lvert \Delta \rvert \leq 92$.
\end{enumerate}

\subsection{The case $h_1 = h_2 = 2h_3$}\label{subsec:twobig}
Since $h_3 \neq h_1, h_2$, we have that $\Delta_3 \neq \Delta_1, \Delta_2$. Then $\Q(x_2) = \Q(x_1 x_3) = \Q(x_1, x_3)$, where the last equality holds by Lemma~\ref{lem:field} since $\Delta_1 \neq \Delta_3$. Hence, $\Q(x_1) \subset \Q(x_2)$. Since $h_1 = h_2$, this is in fact an equality $\Q(x_1) = \Q(x_2)$. 

Suppose $\Delta_1 \neq \Delta_2$. Then $\Q(x_3) = \Q(x_1 x_2) = \Q(x_1, x_2)$ by Lemma~\ref{lem:field}. But then $\Q(x_1) \subset \Q(x_3)$, and so $h_1 = [\Q(x_1) : \Q] \leq [\Q(x_3) : \Q]=h_3$. This though is a contradiction as $h_3 < h_1$ by assumption.

So we must have that $\Delta_1 = \Delta_2$. Note also that $\Q(x_1) = \Q(x_2) \supset \Q(x_3)$. Since $h_1 = 2 h_3$, one therefore has that $[\Q(x_1) : \Q(x_3)]=2$.

If $\Q(\sqrt{\Delta_1})=\Q(\sqrt{\Delta_3})$, then by Lemma~\ref{lem:subfieldsamefund} one has that either $x_3 \in \Q$ or $\Delta_1 \in \{9 \Delta_3 / 4, 4\Delta_3, 9 \Delta_3, 16\Delta_3\}$. The former cannot happen since $h_3 \geq 2$, so we must have that $\Delta_1 \in \{9 \Delta_3 / 4, 4\Delta_3, 9 \Delta_3, 16\Delta_3\}$. Note also that $h_1 \geq 4$ and so certainly $\lvert \Delta_1 \rvert \geq 23$.

Suppose first that $\Delta_1 = \Delta_2 = 9 \Delta_3 /4$ and write $\Delta=\Delta_3$. We may assume that $x_1$ is dominant, and so obtain the lower bound
\begin{align*}
	\lvert \alpha \rvert \geq& (0.9994 e^{3\pi \lvert \Delta \rvert^{1/2}/2}) (\min\{4.4 \times 10^{-5}, 3500 \times \Big (\frac{9}{4}\Big )^{-3} \lvert \Delta \rvert^{-3}\})\\
	&(\min\{4.4 \times 10^{-5}, 3500  \lvert \Delta \rvert^{-3}\}).
\end{align*}

Since $\Q(x_1) = \Q(x_2) \supset \Q(x_3)$, there are exactly $h_1$ conjugates $(x_1', x_2', x_3')$ of $(x_1, x_2, x_3)$. Each conjugate of $x_1, x_2$ occurs exactly once as the coordinate $x_1', x_2'$ respectively of a conjugate $(x_1', x_2', x_3')$. Further, each conjugate $x_3'$ of $x_3$ must appear at least once among the conjugates $(x_1', x_2', x_3')$.

Let $k, m_1, m_2, m_3 \in \Z$ be as given in the following table. When $h_3 \geq k$, we can, by Lemma~\ref{lem:dom} as usual, find conjugates $(x_1', x_2', x_3')$ of $(x_1, x_2, x_3)$, where each $x_i'$ is a singular modulus corresponding to a tuple $(a_i', b_i', c_i') \in T_{\Delta_i}$ with $a_i' \geq m_i$. 

\vspace{0.6cm}

\begin{center}
	\begin{tabular}{  c | c | c | c }\label{tab:conj}
		
		$m_1$ & $m_2$ & $m_3$ & $k$ \\ \hline
		$3$ & $3$ & $3$ & $10$\\
		$4$ & $4$ & $2$ & $12$\\
		$4$ & $4$ & $3$ & $14$\\
		$4$ & $4$ & $4$ & $16$
	\end{tabular}
\end{center}

\vspace{0.6cm}

Each such conjugate gives rise to an upper bound for $\lvert \alpha \rvert$ of the form
\begin{align*}
	\lvert \alpha \rvert &\leq (e^{\pi \lvert \Delta_1 \rvert^{1/2}/m_1}+2079)(e^{\pi \lvert \Delta_2 \rvert^{1/2}/m_2}+2079)(e^{\pi \lvert \Delta_3 \rvert^{1/2}/m_3}+2079)\\
	&= (e^{3\pi \lvert \Delta \rvert^{1/2}/2 m_1}+2079)(e^{3 \pi \lvert \Delta \rvert^{1/2}/2 m_2}+2079)(e^{ \pi \lvert \Delta \rvert^{1/2}/m_3}+2079).
\end{align*}
For $(m_1, m_2, m_3)$ as in the above table, these bounds are incompatible with the earlier lower bound for $\lvert \alpha \rvert$ when $\lvert \Delta \rvert$ is suitably large. Explicitly, we obtain that one of the following holds:
\begin{enumerate}
	\item $2 \leq h_3 \leq 9$;
	\item $10 \leq h_3 \leq 11$ and $\lvert \Delta \rvert \leq 5076$;
	\item $12 \leq h_3 \leq 13$ and $\lvert \Delta \rvert \leq 1430$;
	\item $14 \leq h_3 \leq 15$ and $\lvert \Delta \rvert \leq 255$;
		\item $16 \leq h_3$ and $\lvert \Delta \rvert \leq 164$.
\end{enumerate}

Suppose next that $\Delta_1 = \Delta_2 = 4 \Delta_3$ and write $\Delta=\Delta_3$. We may assume that $x_1$ is dominant, and so obtain the lower bound
\begin{align*}
	\lvert \alpha \rvert \geq& (0.9994 e^{2\pi \lvert \Delta \rvert^{1/2}}) (\min\{4.4 \times 10^{-5}, 3500 \times 4^{-3} \lvert \Delta \rvert^{-3}\})\\
	&(\min\{4.4 \times 10^{-5}, 3500  \lvert \Delta \rvert^{-3}\}).
	\end{align*}

As before, since $\Q(x_1) = \Q(x_2) \supset \Q(x_3)$, there are exactly $h_1$ conjugates $(x_1', x_2', x_3')$ of $(x_1, x_2, x_3)$. Each conjugate of $x_1, x_2$ occurs exactly once as the coordinate $x_1', x_2'$ respectively of a conjugate $(x_1', x_2', x_3')$. Further, each conjugate $x_3'$ of $x_3$ appears at least once among the conjugates $(x_1', x_2', x_3')$.

Let $k, m_1, m_2, m_3 \in \Z$ be as given in the following table. When $h_3 \geq k$, we can by Lemma~\ref{lem:dom} as usual, find conjugates $(x_1', x_2', x_3')$ of $(x_1, x_2, x_3)$, where each $x_i'$ is a singular modulus corresponding to a tuple $(a_i', b_i', c_i') \in T_{\Delta_i}$ with $a_i' \geq m_i$. 

\vspace{0.6cm}

\begin{center}
	\begin{tabular}{  c | c | c | c }\label{tab:conj}
		
		$m_1$ & $m_2$ & $m_3$ & $k$ \\ \hline
		$3$ & $3$ & $3$ & $10$\\
		$3$ & $3$ & $4$ & $12$\\
		$3$ & $3$ & $5$ & $14$\\
		$3$ & $4$ & $4$ & $16$
	\end{tabular}
\end{center}

\vspace{0.6cm}

Each such conjugate gives rise to an upper bound for $\lvert \alpha \rvert$ of the form
\begin{align*}
	\lvert \alpha \rvert &\leq (e^{\pi \lvert \Delta_1 \rvert^{1/2}/m_1}+2079)(e^{\pi \lvert \Delta_2 \rvert^{1/2}/m_2}+2079)(e^{\pi \lvert \Delta_3 \rvert^{1/2}/m_3}+2079)\\
	&= (e^{2\pi \lvert \Delta \rvert^{1/2}/m_1}+2079)(e^{2 \pi \lvert \Delta \rvert^{1/2}/m_2}+2079)(e^{ \pi \lvert \Delta \rvert^{1/2}/m_3}+2079).
\end{align*}
These bounds are incompatible with the above lower bound for $\lvert \alpha \rvert$ when $\lvert \Delta \rvert$ is large. Hence, we must have one of:
\begin{enumerate}
	\item $2 \leq h_3 \leq 9$;
	\item $10 \leq h_3 \leq 11$ and $\lvert \Delta \rvert \leq 650$;
	\item $12 \leq h_3 \leq 13$ and $\lvert \Delta \rvert \leq 317$;
	\item $14 \leq h_3 \leq 15$ and $\lvert \Delta \rvert \leq 236$;
	\item $16 \leq h_3$ and $\lvert \Delta \rvert \leq 129$.
\end{enumerate}

Now suppose that $\Delta_1 = \Delta_2 = 9 \Delta_3$ and write $\Delta = \Delta_3$. We may assume that $x_1$ is dominant, and so obtain the lower bound
\begin{align*}
	\lvert \alpha \rvert \geq& (0.9994 e^{3\pi \lvert \Delta \rvert^{1/2}}) (\min\{4.4 \times 10^{-5}, 3500 \times 9^{-3} \lvert \Delta \rvert^{-3}\})\\
	&(\min\{4.4 \times 10^{-5}, 3500  \lvert \Delta \rvert^{-3}\}).
\end{align*}

As before, there are exactly $h_1$ conjugates $(x_1', x_2', x_3')$ of $(x_1, x_2, x_3)$ and each conjugate of $x_1, x_2$ occurs exactly once as the coordinate $x_1', x_2'$ respectively of a conjugate $(x_1', x_2', x_3')$. Further, each conjugate $x_3'$ of $x_3$ appears at least once among the conjugates $(x_1', x_2', x_3')$.

As previously, when $h_3 \geq k$, we can find conjugates $(x_1', x_2', x_3')$ of $(x_1, x_2, x_3)$, where each $x_i'$ is a singular modulus corresponding to a tuple $(a_i', b_i', c_i') \in T_{\Delta_i}$ with $a_i' \geq m_i$, and $k, m_1, m_2, m_3 \in \Z$ are as given in the following table. 

\vspace{0.6cm}

\begin{center}
	\begin{tabular}{  c | c | c | c }\label{tab:conj}
		
		$m_1$ & $m_2$ & $m_3$ & $k$ \\ \hline
		$3$ & $3$ & $2$ & $8$\\
		$3$ & $4$ & $2$ & $10$
	\end{tabular}
\end{center}

\vspace{0.6cm}

Each such conjugate gives rise to an upper bound for $\lvert \alpha \rvert$ of the form
\begin{align*}
	\lvert \alpha \rvert &\leq (e^{3\pi \lvert \Delta \rvert^{1/2}/m_1}+2079)(e^{3 \pi \lvert \Delta \rvert^{1/2}/m_2}+2079)(e^{ \pi \lvert \Delta \rvert^{1/2}/m_3}+2079).
\end{align*}
For $(m_1, m_2, m_3)$ as in the above table, these bounds are incompatible with the lower bound for $\lvert \alpha \rvert$ when $\lvert \Delta \rvert$ is large. One thus has that:
\begin{enumerate}
	\item $2 \leq h_3 \leq 7$;
	\item $8 \leq h_3 \leq 9$ and $\lvert \Delta \rvert \leq 255$;
	\item $10 \leq h_3$ and $\lvert \Delta \rvert \leq 85$.
\end{enumerate}

Finally suppose that $\Delta_1 = \Delta_2 = 16 \Delta_3$ and write $\Delta = \Delta_3$. We may assume that $x_1$ is dominant, and so obtain the lower bound
\begin{align*}
	\lvert \alpha \rvert \geq& (0.9994 e^{4\pi \lvert \Delta \rvert^{1/2}}) (\min\{4.4 \times 10^{-5}, 3500 \times 16^{-3} \lvert \Delta \rvert^{-3}\})\\
	&(\min\{4.4 \times 10^{-5}, 3500  \lvert \Delta \rvert^{-3}\}).
\end{align*}

As before, there are exactly $h_1$ conjugates $(x_1', x_2', x_3')$ of $(x_1, x_2, x_3)$ and each conjugate of $x_1, x_2$ occurs exactly once as the coordinate $x_1', x_2'$ respectively of a conjugate $(x_1', x_2', x_3')$. Further, each conjugate $x_3'$ of $x_3$ appears at least once among the conjugates $(x_1', x_2', x_3')$.

When $h_3 \geq k$, we can then by Lemma~\ref{lem:dom} as usual, find conjugates $(x_1', x_2', x_3')$ of $(x_1, x_2, x_3)$, where each $x_i'$ is a singular modulus corresponding to a tuple $(a_i', b_i', c_i') \in T_{\Delta_i}$ with $a_i' \geq m_i$, and $k, m_1, m_2, m_3 \in \Z$ are as given in the following table. 

\vspace{0.6cm}

\begin{center}
	\begin{tabular}{  c | c | c | c }\label{tab:conj}
		
		$m_1$ & $m_2$ & $m_3$ & $k$ \\ \hline
		$3$ & $3$ & $2$ & $8$\\
		$3$ & $3$ & $3$ & $10$
	\end{tabular}
\end{center}

\vspace{0.6cm}

Each such conjugate gives rise to an upper bound for $\lvert \alpha \rvert$ of the form
\begin{align*}
	\lvert \alpha \rvert &\leq (e^{4\pi \lvert \Delta \rvert^{1/2}/m_1}+2079)(e^{4 \pi \lvert \Delta \rvert^{1/2}/m_2}+2079)(e^{ \pi \lvert \Delta \rvert^{1/2}/m_3}+2079).
\end{align*}
For $(m_1, m_2, m_3)$ as in the above table, these bounds are incompatible with the earlier lower bound for $\lvert \alpha \rvert$ when $\lvert \Delta \rvert$ is suitably large. Explicitly, we obtain that one of the following holds:
\begin{enumerate}
	\item $2 \leq h_3 \leq 7$;
	\item $8 \leq h_3 \leq 9$ and $\lvert \Delta \rvert \leq 79$;
	\item $10 \leq h_3$ and $\lvert \Delta \rvert \leq 52$.
\end{enumerate}

Now consider the case when $\Q(\sqrt{\Delta_1}) \neq \Q(\sqrt{\Delta_3})$. Then by Lemma~\ref{lem:subfielddiffund} one of the following holds:
\begin{enumerate}
	\item at least one of $\Delta_1$ or $\Delta_3$ is listed in \cite[Table~2.1]{AllombertBiluMadariaga15};
	\item  $h_1 \geq 128$.
\end{enumerate}
If $\Delta_i$ is listed in \cite[Table~2.1]{AllombertBiluMadariaga15}, then we can find all possibilities for $(\Delta_1, \Delta_2, \Delta_3)$ satisfying the condition $[\Q(x_1) : \Q(x_3)] =2$.

So suppose $h_1 \geq 128$. Write $\Delta = \max\{\lvert \Delta_1 \rvert, \lvert \Delta_2 \rvert, \lvert \Delta_3 \rvert \}$. Taking conjugates as necessary, we may assume that $x_i$ is dominant, where $\Delta = \lvert \Delta_i \rvert$. Since $h_i \geq 64$, certainly $\lvert \Delta_i \rvert \geq 23$. Then by the bounds in Subsection~\ref{subsec:bounds}
\[\lvert \alpha \rvert \geq (0.9994 e^{\pi \Delta^{1/2}}) (\min\{4.4 \times 10^{-5}, 3500  \Delta^{-3}\})^2.\]

Since $h_1 \geq 128$, we can always find a conjugate $(x_1', x_2', x_3')$ of $(x_1, x_2, x_3)$ with the associated $a_1', a_2', a_3'$ satisfying $a_1', a_2' \geq 4$ and $a_3' \geq 5$. This gives rise to the upper bound
\[ \lvert \alpha \rvert \leq (e^{ \pi \Delta^{1/2}/4}+2079) (e^{\pi \Delta^{1/2}/4}+2079) (e^{\pi \Delta^{1/2}/5}+2079).\]
Together these bounds imply that $\Delta \leq 488$. Hence $h_1 \geq 128$ and $\lvert \Delta_1 \rvert, \lvert \Delta_3 \rvert \leq 488$.

\subsection{The case $h_1 = 2h_2 = 2 h_3$}\label{subsec:onebig}
Since $h_1 \neq h_2, h_3$, one has that $\Delta_1 \neq \Delta_2, \Delta_3$. Therefore, $\Q(x_3) = \Q(x_1x_2) = \Q(x_1, x_2)$. The last equality holds by Lemma~\ref{lem:field} since $\Delta_1 \neq \Delta_2$. Thus $\Q(x_1) \subset \Q(x_3)$ and so $h_1 = [\Q(x_1) : \Q] \leq [\Q(x_3) : \Q]=h_3$. This though is a contradiction as $h_3 < h_1$ by assumption, and so we may eliminate this case.

\section{Eliminating non-trivial cases}\label{sec:elim}

Recall that we assumed $x_1, x_2, x_3$ are singular moduli with $x_1 x_2 x_3 = \alpha \in \Q^{\times}$. We write $\Delta_i$ for their respective discriminants and $h_i$ for the corresponding class numbers $h(\Delta_i)$. Without loss of generality $h_1 \geq h_2 \geq h_3$. Assuming that we are not in one of the trivial cases (1)--(3) of Theorem~\ref{thm:main}, then we have shown in Section~\ref{sec:pf} that $x_1, x_2, x_3$ are pairwise distinct and that we must be in one of the following cases.

\begin{myEnumerate}
	\item $h_1=h_2=h_3$. \\
	Write $h = h_1 = h_2 = h_3$.      
	\begin{myEnumerate}
		\item $\Delta_1 = \Delta_2 = \Delta_3$. \\
		 Write $\Delta = \Delta_1 = \Delta_2 = \Delta_3$. 
		 \begin{myEnumerate}
		 	\item $4 \leq h \leq 11$.
		 	\item $12 \leq h \leq 13$ and $\lvert \Delta \rvert \leq 30339$.
		 	\item $14 \leq h \leq 15$ and $\lvert \Delta \rvert \leq 4124$.
		 	\item $16 \leq h \leq 17$ and $\lvert \Delta \rvert \leq 1045$.
		 	\item $18 \leq h \leq 19$ and $\lvert \Delta \rvert \leq 488$.
		 	\item $20 \leq h$ and $\lvert \Delta \rvert \leq 334$.
		 \end{myEnumerate}
	 
		\item $\lvert \Delta_1 \rvert > \lvert \Delta_2 \rvert$. \\
		In this case, $\Q(x_1)= \Q(x_2)= \Q(x_3)$.
		\begin{myEnumerate}
			\item $\Q(\sqrt{\Delta_i}) \neq \Q(\sqrt{\Delta_j})$ for some $i, j$.\\
			 The list of all possible $\Delta_1, \Delta_2, \Delta_3$ is given in \cite[Table~4.1]{AllombertBiluMadariaga15}.
			\item $\Q(\sqrt{\Delta_1})=\Q(\sqrt{\Delta_2})=\Q(\sqrt{\Delta_3})$.\\
			In this case, $\Delta_1 = 4 \Delta_2$, $\Delta_3 \in \{\Delta_1, \Delta_2\}$, and $\Delta_2 \equiv 1 \bmod 8$.
			\begin{myEnumerate}
				\item $\Delta_3 = \Delta_2$.\\
				Write $\Delta = \Delta_2 = \Delta_3$.
				\begin{myEnumerate}
					\item $2 \leq h \leq 3$.
					\item $4 \leq h \leq 5$ and $\lvert \Delta \rvert \leq 367$.
					\item $6 \leq h \leq 7$ and $\lvert \Delta \rvert \leq 163$.
					\item $8 \leq h$ and $\lvert \Delta \rvert \leq 93$.
					\end{myEnumerate}
				\item $\Delta_3 = \Delta_1$.\\
				Write $\Delta = \Delta_2$.
				\begin{myEnumerate}
					\item $2 \leq h \leq 3$.
					\item $4 \leq h \leq 5$ and $\lvert \Delta \rvert \leq 5781$.
					\item $6 \leq h \leq 7$ and $\lvert \Delta \rvert \leq 650$.
					\item $8 \leq h \leq 9$ and $\lvert \Delta \rvert \leq 192$.
					\item $10 \leq h$ and $\lvert \Delta \rvert \leq 92$.
					\end{myEnumerate}
				\end{myEnumerate}
			\end{myEnumerate}
		\end{myEnumerate}
	
\item $h_1=h_2=2h_3$.\\
In this case, $\Q(x_1)=\Q(x_2)\supset \Q(x_3)$ and $[\Q(x_1) : \Q(x_3)] =2$. 
\begin{myEnumerate}
	\item $\Delta_1 \neq \Delta_2$.\\
	This case cannot arise.
	\item $\Delta_1 = \Delta_2$.
	\begin{myEnumerate}
		\item $\Q(\sqrt{\Delta_1})=\Q(\sqrt{\Delta_3})$.
		\begin{myEnumerate}
			\item $\Delta_1 = \Delta_2 = 9 \Delta_3 /4$.\\
				Write $\Delta = \Delta_3$.
			\begin{myEnumerate}
				\item $2 \leq h_3 \leq 9$.
				\item $10 \leq h_3 \leq 11$ and $\lvert \Delta \rvert \leq 5076$.
				\item $12 \leq h_3 \leq 13$ and $\lvert \Delta \rvert \leq 1430$.
				\item $14 \leq h_3 \leq 15$ and $\lvert \Delta \rvert \leq 255$.
				\item $16 \leq h_3$ and $\lvert \Delta \rvert \leq 164$.
			\end{myEnumerate}
			\item $\Delta_1 = \Delta_2 = 4 \Delta_3$.\\
			Write $\Delta = \Delta_3$.
			\begin{myEnumerate}
				\item $2 \leq h_3 \leq 9$;
			\item $10 \leq h_3 \leq 11$ and $\lvert \Delta \rvert \leq 650$.
				\item $12 \leq h_3 \leq 13$ and $\lvert \Delta \rvert \leq 317$.
				\item $14 \leq h_3 \leq 15$ and $\lvert \Delta \rvert \leq 236$.
				\item $16 \leq h_3$ and $\lvert \Delta \rvert \leq 129$.
			\end{myEnumerate}
		\item $\Delta_1 = \Delta_2 = 9 \Delta_3$.\\
			Write $\Delta = \Delta_3$.
		\begin{myEnumerate}
			\item $2 \leq h_3 \leq 7$.
			\item $8 \leq h_3 \leq 9$ and $\lvert \Delta \rvert \leq 255$.
			\item $10 \leq h_3$ and $\lvert \Delta \rvert \leq 85$.
		\end{myEnumerate}
			\item $\Delta_1 = 16 \Delta_3$.\\
			Write $\Delta = \Delta_3$.
			\begin{myEnumerate}
				\item $2 \leq h_3 \leq 7$.
				\item $8 \leq h_3 \leq 9$ and $\lvert \Delta \rvert \leq 79$.
				\item $10 \leq h_3$ and $\lvert \Delta \rvert \leq 52$.
			\end{myEnumerate}
		\end{myEnumerate}
	\item $\Q(\sqrt{\Delta_1}) \neq \Q(\sqrt{\Delta_3})$.
	\begin{myEnumerate}
	\item	$\Delta_1$ is listed in \cite[Table~2.1]{AllombertBiluMadariaga15}.
	\item	$\Delta_3$ is listed in \cite[Table~2.1]{AllombertBiluMadariaga15}.
	\item $h_1 \geq 128$ and $\lvert \Delta_1 \rvert, \lvert \Delta_3 \rvert \leq 488$.
	\end{myEnumerate}
		\end{myEnumerate}
\end{myEnumerate}
\item $h_1 = 2 h_2 = 2h_3$.\\
This case cannot arise.
\end{myEnumerate}

The finite list of discriminants $(\Delta_1, \Delta_2, \Delta_3)$ satisfying one of these conditions may be computed in PARI. In fact, there are $2888$ triples $(\Delta_1, \Delta_2, \Delta_3)$ satisfying one of the above conditions. We now show how each such choice of $(\Delta_1, \Delta_2, \Delta_3)$ may be eliminated by another computation in PARI.

For each tuple $(\Delta_1, \Delta_2, \Delta_3)$ satisfying one of the above conditions, we show that $x_1 x_2 x_3 \notin \Q$, for any choice of $x_1, x_2, x_3$ pairwise distinct singular moduli of respective discriminant $\Delta_i$. (In fact, by taking conjugates as necessary, it is enough to eliminate all possible choices of $x_2, x_3$ for some fixed $x_1$.) To do this, we use the following algorithm.

For each possible choice of $(x_1, x_2, x_3)$, let $L$ be a number field containing all conjugates of $x_1, x_2, x_3$. If $x_1 x_2 x_3 \in \Q^{\times}$, then
\[\frac{x_1 x_2 x_3}{\sigma(x_1) \sigma(x_2) \sigma(x_3)} =1\]
for every automorphism $\sigma \in \gal(L/\Q)$. It therefore suffices to find an automorphism of $L$ without this property in order to eliminate the tuple $(x_1, x_2, x_3)$. Once all such tuples $(x_1, x_2, x_3)$ have been eliminated, we can eliminate the tuple $(\Delta_1, \Delta_2, \Delta_3)$. 

It remains to find a suitable field $L$. In 1(a), the $x_i$ are all conjugate, so $L= \Q(\sqrt{\Delta_1}, x_1)$ suffices. In 1(b)(i), the field $\Q(x_1)$ is Galois, and so we may take $L= \Q(x_1)$. In 1(b)(ii), $L= \Q(\sqrt{\Delta_1}, x_1)$ works. In 2(b)(i), let $L = \Q(\sqrt{\Delta_1}, x_1)$. In 2(b)(ii)(A), the field $\Q(x_1)$ is Galois and so we may take $L = \Q(x_1)$. In 2(b)(ii)(B), the field $\Q(x_3)$ is Galois and so $L = \Q(\sqrt{\Delta_1}, x_1)$ suffices. There are no $(\Delta_1, \Delta_2, \Delta_3)$ satisfying the conditions in 2(b)(ii)(C), so this case may be excluded.

We implement the resulting algorithm using a PARI script. Running it, we are able to eliminate each of the above possibilities. The proof of Theorem~\ref{thm:main} is thus complete. Total running time of our program is about 12 hours on a standard laptop computer.\footnote{With a 2.5GHz Intel i5 processor and 8GB RAM.} The time taken to find and eliminate a triple of discriminants $(\Delta_1, \Delta_2, \Delta_3)$ satisfying one of the above conditions increases with the respective class numbers $h_1, h_2, h_3$. Consequently, approximately 75\% of the overall run time of the program is spent dealing with case (2)(b)(ii)(B), which includes discriminants $\Delta_1$ of class number 32, greater than in any other case. 

\bibliographystyle{amsplain}
\bibliography{refs}

\providecommand{\bysame}{\leavevmode\hbox to3em{\hrulefill}\thinspace}
\providecommand{\MR}{\relax\ifhmode\unskip\space\fi MR }
\providecommand{\MRhref}[2]{%
  \href{http://www.ams.org/mathscinet-getitem?mr=#1}{#2}
}
\providecommand{\href}[2]{#2}
\begin{thebibliography}{10}

\bibitem{AllombertBiluMadariaga15}
B.~Allombert, Yu. Bilu, and A.~Pizarro-Madariaga, \emph{C{M}-points on straight
  lines}, Analytic number theory, Springer, Cham, 2015, pp.~1--18.

\bibitem{Andre98}
Y.~Andr\'{e}, \emph{Finitude des couples d'invariants modulaires singuliers sur
  une courbe alg\'{e}brique plane non modulaire}, J. Reine Angew. Math.
  \textbf{505} (1998), 203--208.

\bibitem{BiluKuhne18}
Yu. Bilu and L.~K\"{u}hne, \emph{Linear equations in singular moduli}, IMRN
  (2018).

\bibitem{BiluLucaMadariaga16}
Yu. Bilu, F.~Luca, and A.~Pizarro-Madariaga, \emph{Rational products of
  singular moduli}, J. Number Theory \textbf{158} (2016), 397--410.
  

\bibitem{BiluMasserZannier13}
Yu. Bilu, D.~Masser, and U.~Zannier, \emph{An effective ``theorem of
  {A}ndr\'{e}'' for {$CM$}-points on a plane curve}, Math. Proc. Cambridge
  Philos. Soc. \textbf{154} (2013), no.~1, 145--152. 

\bibitem{Binyamini19}
G.~Binyamini, \emph{Some effective estimates for {A}ndr{\'e}--{O}ort in
  {$Y(1)^n$}}, J. Reine Angew. Math. (2019).

\bibitem{Cox89}
D.~Cox, \emph{Primes of the form {$x^2 + ny^2$}}, John Wiley \& Sons, Inc., New
  York, 1989. 

\bibitem{FayeRiffaut18}
B.~Faye and A.~Riffaut, \emph{Fields generated by sums and products of singular
  moduli}, J. Number Theory \textbf{192} (2018), 37--46. 

\bibitem{Kuhne12}
L.~K{\"{u}}hne, \emph{An effective result of {A}ndr\'{e}-{O}ort type}, Ann. of
  Math. (2) \textbf{176} (2012), no.~1, 651--671. 

\bibitem{PARI2}
{PARI~Group}, Univ. Bordeaux, \emph{{PARI/GP version {\tt 2.11.4}}}, 2020,
  available from \url{http://pari.math.u-bordeaux.fr/}.

\bibitem{Pila11}
J.~Pila, \emph{O-minimality and the {A}ndr\'e-{O}ort conjecture for {$\mathbb{
  C}^n$}}, Ann. of Math. (2) \textbf{173} (2011), no.~3, 1779--1840.
  
\bibitem{PilaTsimerman17}
J.~Pila and J.~Tsimerman, \emph{Multiplicative relations among singular
  moduli}, Ann. Sc. Norm. Super. Pisa Cl. Sci. (5) \textbf{17} (2017), no.~4,
  1357--1382. 

\bibitem{Riffaut19}
A.~Riffaut, \emph{Equations with powers of singular moduli}, Int. J. Number
  Theory \textbf{15} (2019), no.~3, 445--468. 

\bibitem{sagemath}
The Sage~Developers, \emph{{S}agemath, the {S}age {M}athematics {S}oftware
  {S}ystem ({V}ersion 9.0)}, 2020, \url{https://www.sagemath.org}.

\end{thebibliography}

\end{document}